\documentclass{amsart}
\usepackage{amsfonts,amssymb,amscd,amsmath,enumerate,verbatim,calc}
\usepackage[all]{xy}

\newcommand{\CM}{Cohen-Macaulay}

\newcommand{\n}{\mathfrak{n} }
\newcommand{\m}{\mathfrak{m} }

\newcommand{\q}{\mathfrak{q} }

\newcommand{\C}{\mathcal{C} }
\newcommand{\Z}{\mathbb{Z} }

\newcommand{\rt}{\rightarrow}

\newcommand{\bx}{\mathbf{x} }
\newcommand{\Xb}{\mathbf{X}_\bullet}
\newcommand{\Zb}{\mathbf{Z}_\bullet}
\newcommand{\Yb}{\mathbf{Y}_\bullet}
\newcommand{\Kb}{\mathbf{K}_\bullet}

\newcommand{\Supp}{\operatorname{Supp}}
\newcommand{\grade}{\operatorname{grade}}

\newcommand{\height}{\operatorname{height}}
\newcommand{\ann}{\operatorname{ann}}
\newcommand{\proj}{\operatorname{proj}}
\newcommand{\projdim}{\operatorname{projdim}}

\newcommand{\cone}{\operatorname{cone}}
\newcommand{\Tor}{\operatorname{Tor}}

\newcommand{\I}{\operatorname{Iso}}
\newcommand{\Hom}{\operatorname{Hom}}
\newcommand{\Ext}{\operatorname{Ext}}

\theoremstyle{plain}

\newtheorem{theorem}{Theorem}[section]
\newtheorem{corollary}[theorem]{Corollary}
\newtheorem{lemma}[theorem]{Lemma}

\theoremstyle{definition}

\newtheorem{s}[theorem]{}

\theoremstyle{remark}

\begin{document}

\title[Derived functors]{Derived functors and Hilbert polynomials over regular local rings}
\author{Tony~J.~Puthenpurakal}
\date{\today}
\address{Department of Mathematics, IIT Bombay, Powai, Mumbai 400 076}

\email{tputhen@math.iitb.ac.in}

 \begin{abstract}
Let $(A,\mathfrak{m})$ be a regular local ring of dimension $d \geq 1$, $I$ an $\mathfrak{m}$-primary ideal. Let $N$ be a non-zero finitely generated $A$-module. Consider the functions
\[
t^I(N, n) = \sum_{i = 0}^{ d}\ell(\text{Tor}^A_i(N, A/I^n)) \ \text{and}\  e^I(N, n) = \sum_{i = 0}^{ d}\ell(\text{Ext}_A^i(N, A/I^n))
\]
of polynomial type and let their degrees be $t^I(N) $ and $e^I(N)$. We prove
that $t^I(N) = e^I(N) =  \max\{ \dim N, d -1 \}$.
\end{abstract}
 \maketitle
\section{introduction}
In this paper all rings considered are commutative, Noetherian, local with unity and all modules considered will be finitely generated. Let $(A, \m)$ be a local ring of dimension $d \geq 1$, $I$ an $\m$-primary  ideal in $A$ and let $L$ be an $A$-module. If $T$ is an $A$-module of finite length then we denote by $\ell(T)$ its length. The Hilbert-Samuel polynomial $n \mapsto \ell(L/I^nL)$ of $L$ with respect to $I$
is well-studied. It is known that it is of polynomial type and of degree $\dim L$. Considerably less is known of the function $n \mapsto \ell(\Tor^A_i(L, A/I^n))$ for $i \geq 1$. It is known that this function is of polynomial type and of degree $\leq d - 1$. There are some results which show under certain conditions the maximal degree is attained, see \cite{P}, \cite{IP} and \cite{KT}. However this function can also be identically zero, see \cite[Remark 20]{P}. Similarly not much is known of the function  $n \mapsto \ell(\Ext^i_A(L, A/I^n))$ for $i \geq 1$.  It is known that this function is of polynomial type and  of degree $\leq d - 1$.
There are some results which show under certain conditions the maximal degree is attained, see \cite{KP}, \cite{CKST}. Even less is known of the functions $n \mapsto \ell(\Tor^A_i(L, M/I^nM))$ and $n \mapsto \ell(\Ext_A^i(L, M/I^nM))$ where $M$ is an $A$-module.

Perhaps the first case to consider for these functions is when $A$ is regular. In this case $\projdim N$ is finite for any $A$-module $N$. Surprisingly we found out that the functions
\[
t^I(N, n) = \sum_{i = 0}^{ d}\ell(\Tor^A_i(N, A/I^n)) \ \text{and}\  e^I(N, n) = \sum_{i = 0}^{ d}\ell(\Ext_A^i(N, A/I^n))
\]
are \emph{easier} to tackle. One can then work with $K^b(\proj A)$, the homotopy category of bounded complexes of projective $A$-modules,  which is the bounded derived category of $A$.
More generally  let $(A,\m)$ be a local ring (not necessarily regular). Let $\Xb \colon \Xb^{-1} \rt \Xb^0 \rt \Xb^{1}$ be a complex of $A$-modules. In \cite[Proposition 3]{T}  it is shown that if $\ell(H^0(\Xb\otimes M/I^nM))$ has finite length for all $n \geq 1$ then the  function $n \rt \ell(H^0(\Xb\otimes M/I^nM))$ is of polynomial type. The precise degree of this polynomial is difficult to determine (a general upper bound for the degree is given in \cite[Proposition 3]{T}).

\s \label{setup} In this paper we prove a surprising  result. Let $(A, \m)$ be a local ring and let  $K^b(\proj A)$ be the homotopy category of bounded complexes of projective $A$-modules, Let $K^b_f(\proj A)$
denote the homotopy category of bounded complexes of projective $A$-modules with finite length cohomology. Let $\Xb \in K^b_f(\proj A)$. We note that for any $A$-module $M$ and an ideal $I$ we have
$\ell(H^i(\Xb \otimes M/I^nM))$ has finite length for all $n \geq 1$ and for all $i \in \Z$. The main point of this paper is that it is better to look at the function
\[
\psi_{\Xb}^{M, I}(n) = \sum_{i \in \Z}\ell(H^i(\Xb \otimes M/I^nM)), \quad  \text{for $n \geq 1$}.
\]
We know that $\psi_{\Xb}^{M, I}(n)$ is of polynomial type say of degree $r_I^M(\Xb)$.
The main result of this paper is
\begin{theorem}\label{main}[with hypotheses as in \ref{setup}] Assume $M \neq 0$ and $I \neq A$. Then there exists a non-negative integer $r_I^M$ depending only on $I$ and $M$ such if $\Xb \in K^b_f(\proj A)$ is non-zero then $r_I^M(\Xb) = r_I^M$.
\end{theorem}
The essential reason why this happens is because $K^b_f(\proj A)$ has \emph{no} proper thick subcategories.
\s Thus to determine $r_I^M(X)$ it suffices to compute it for a single non-zero complex $\Xb$ in $K^b_f(\proj A)$. As a consequence of Theorem \ref{main} we show
\begin{theorem}
\label{max-dim}[with hypotheses as in Theorem \ref{main}] If $\dim M > 0$ and $I$ is $\m$-primary then $r_I^M = \dim M - 1$.
\end{theorem}

\s \label{reg-setup} Let $A$ be a \CM \ local ring. Let $I \neq A$ be an ideal of $A$ and let $M$ be a non-zero $A$-module. If $L$ is a non-zero module of finite length and finite projective dimension, set
\[
t^I_M(L, n) = \sum_{i = 0}^{ \dim A}\ell(\Tor^A_i(L, M/I^nM)) \ \text{and}\  e^I_M(L, n) = \sum_{i = 0}^{ \dim A}\ell(\Ext_A^i(L, M/I^nM)).
\]
Also let $t^I_M(L)$ and $e^I_M(L)$ denote the degree of the corresponding functions of polynomial type. We show
\begin{corollary}\label{reg-cor}
(with hypotheses as in \ref{reg-setup}) Let $L_1, L_2$ be two non-zero modules of finite length and finite projective dimension. Then
\[
t^I_M(L_1) = t^I_M(L_2) = e^I_M(L_1) = e^I_M(L_2).
\]
\end{corollary}

\s\label{m-prim-setup}  We now consider the case when $\dim M > 0$ and $I$ is $\m$-primary. Then by \cite[Proposition 3]{T} it follows that $\psi_{\Xb}^{M, I}(n)$ is of degree
\[
s_I^M(\Xb)  \leq \max \{ \dim H^*(\Xb \otimes  M) , \dim M - 1 \}.
\]
Furthermore if $\dim H^*(\Xb\otimes M) \geq \dim M$ then $s_I^M(\Xb) = \dim H^*(\Xb \otimes  M)$.
We prove
\begin{theorem}\label{m-prim-bound}(with hypotheses as in \ref{m-prim-setup}) We have
\[
s_I^M(\Xb)  =  \max \{ \dim H^*(\Xb \otimes  M) , \dim M - 1 \}.
\]
\end{theorem}
\s \label{CM-setup} Let $I \neq A$ be an $\m$-primary ideal of $A$ and let $M$ be a  $A$-module with $\dim M > 0$. If $L$ is a non-zero module of  finite projective dimension, set
\[
t^I_M(L, n) = \sum_{i = 0}^{ \dim A}\ell(\Tor^A_i(L, M/I^nM)) \ \text{and}\  e^I_M(L, n) = \sum_{i = 0}^{ \dim A}\ell(\Ext_A^i(L, M/I^nM)).
\]
Also let $t^I_M(L)$ and $e^I_M(L)$ denote the degree of the corresponding functions of polynomial type. As an application of Theorem \ref{m-prim-bound} we have
\begin{corollary}\label{cm-max}( with hypotheses as in \ref{CM-setup}) We have
\[
t^I_M(L) = e^I_M(L) =  \max \{ \dim M\otimes L , \dim M - 1 \}.
\]
\end{corollary}
As an application of this corollary  we get the result stated in the abstract, i.e.,
\begin{corollary}( with hypotheses as in \ref{CM-setup}) Assume $A$ is regular and $M = A$. We have
\[
t^I(L) = e^I(L) =  \max \{ \dim  L , \dim A - 1 \}.
\]
\end{corollary}

We now describe in brief the contents of this paper. In section two we discuss a few preliminary results. In section three we prove Theorem \ref{main} and Corollary \ref{reg-cor}. In section four we give a proof of Theorem \ref{max-dim}. In section five we give a proof of Theorem \ref{m-prim-bound}. Finally in section six we give a proof of Corollary \ref{cm-max}.
\section{Preliminaries}
In this section we discuss a few preliminary results that we need.
We use \cite{N} for notation on triangulated categories. However we will assume that if $\mathcal{C}$ is a triangulated category then $\Hom_\mathcal{C}(X, Y)$ is a set for any objects $X, Y$ of $\mathcal{C}$.

\s \label{t-f} Let $\C$ be an essentially small triangulated category  with shift operator $\Sigma$ and let $\I(\C)$ be the set of isomorphism classes of objects in $\C$. By a \emph{weak triangle function} on $\C$ we mean a function $\xi \colon \I(\C) \rt \Z$ such that
\begin{enumerate}
  \item $\xi(X) \geq 0$ for all $X \in \C$.
  \item $\xi(0) = 0$.
  \item $\xi(X \oplus Y) = \xi(X) + \xi(Y)$ for all $X, Y \in \C$.
  \item $\xi(\Sigma X ) = \xi(X)$ for all $X \in \C$.
  \item If $X \rt Y \rt Z \rt \Sigma X $ is a triangle in $\C$ then
   $\xi(Z) \leq \xi(X) + \xi(Y)$.
\end{enumerate}
\s Set $$\ker \xi = \{ X \mid \xi(X) = 0 \}.$$
The following result (essentially an observation) is a crucial ingredient in our proof of Theorem \ref{main}.
\begin{lemma}
\label{ker-lemma}(with hypotheses as above)
$\ker \xi $ is a thick subcategory of $\C$.
\end{lemma}
\begin{proof}
  We have
  \begin{enumerate}
    \item $0 \in \ker \xi$.
    \item If $X \cong Y$ and $X \in \ker \xi$. Then note $\xi(Y) = \xi(X) = 0$. So $Y \in \ker \xi$.
    \item If $X \in \ker \xi$ then note $\xi(\Sigma X) = \xi(X) = 0$. So $\Sigma X \in \ker \xi$. Similarly $\Sigma^{-1} X \in \ker \xi$.
    \item If $X \rt Y \rt Z \rt \Sigma X$ is a triangle in $\C$ with $X, Y \in \ker \xi$. Then note
    \[
    0 \leq \xi(Z) \leq \xi(X) + \xi(Y) = 0 + 0 = 0.
    \]
    So $Z \in \ker \xi$.
    \item
    If $X \oplus Y \in \ker \xi$ then $\xi(X) + \xi(Y) = \xi(X \oplus Y) = 0$. As $\xi(X), \xi(Y)$ are non-negative it follows that $\xi(X) = \xi(Y) = 0$. Thus $X, Y \in \ker \xi$.
  \end{enumerate}
  It follows that $\ker \xi$ is a thick subcategory of $\C$.
\end{proof}

\s Let $A$ be a ring. Let $K^b(\proj A)$ be  the homotopy category of bounded complexes of projective complexes.
We index complexes cohomologically
$$\Xb \colon  \cdots \rt \Xb^{n-1} \rt \Xb^n \rt \Xb^{n+1} \rt \cdots.$$
We note that $\Xb = 0$  in $K^b(\proj A)$ if and only if $H^*(\Xb) = 0$. If $\Xb = 0$  in $K^b(\proj A)$ then note that $H^*(X\otimes N) = 0$ for any $A$-module $N$.

\s \label{fl} Let $K^b_f(\proj A)$ denote the homotopy category of bounded complexes of projective complexes with finite length cohomology. We note that if $\Xb \in K^b_f(\proj A)$ and $N$ is an $A$-module
then $H^*(\Xb \otimes N)$ also has finite length. To see this if $P$ is a prime ideal in $A$ with $P \neq \m$ then
\[
H^*(\Xb \otimes_A N)_P = H^*({\Xb}_P \otimes_{A_P} N_P)  = 0 \quad \text{as ${\Xb}_P = 0$ in $K^b(\proj A_P)$}.
\]

\begin{lemma}
  \label{non-zero} Let $\Xb \in K^b(\proj A)$ be non-zero. Let $N \neq 0$. Then $H^*(\Xb \otimes N) \neq 0$.
\end{lemma}
\begin{proof}
  We may assume $\Xb$ is a minimal complex. Furthermore (after a shift) we may assume that $\Xb^0 \neq 0$ and $\Xb^i = 0$ for $i \geq 1$.
  Let $H^0(\Xb) = E \neq 0$ since $\Xb$ is minimal. It is straight forward to check that $H^0(\Xb \otimes N) = E\otimes N \neq 0$. The result follows.
\end{proof}

\s \label{deg-bound} Suppose for an $A$-module $M$ and an ideal $I$ we have
$\ell(H^i(\Xb \otimes M/I^nM))$ has finite length for all $n \geq 1$ and for all $i \in \Z$. Consider the function
\[
\psi_{\Xb}^{M, I}(n) = \sum_{i \in \Z}\ell(H^i(\Xb \otimes M/I^nM)), \quad  \text{for $n \geq 1$}.
\]
By \cite[Proposition 3]{T}  we know that $\psi_{\Xb}^{M, I}(n)$ is of polynomial type say of degree $r_I^M(X)$ and
\[
r_I(M) \leq \max\{ \dim H^*(\Xb \otimes M), \ell_M(I) - 1 \},
\]
Furthermore if $\dim H^*(\Xb \otimes M) \geq \ell_M(I)$ then equality holds above. Here $\ell_M(I)$ denotes the analytic spread of $I$ with respect to $M$, i.e. the dimension of the module $R(I, M) \otimes_A k$ over the Rees algebra $R(I)$ of $I$  (where $R(I,M)  = \bigoplus_{n \geq 0}I^nM$ denotes the Rees module of $M$ with respect to $I$). If $I$ is $\m$-primary then note $\ell_I(M) = \dim M$.
\section{Proof of Theorem \ref{main} and Corollary \ref{reg-cor}}
In this section we give proofs of Theorem \ref{main} and Corollary \ref{reg-cor}. We first give
\begin{proof}[Proof of Theorem \ref{main}]
By \ref{non-zero} it follows that the function $\psi_{\Xb}^{M, I}(n) \neq 0$ for all $n \geq 1$. Thus $r_I^M(\Xb) \geq 0$ for all $\Xb \neq 0$.
Also by \ref{deg-bound},  $r_I^M(\Xb) \leq \dim A$ for any $\Xb \in K^b_f(\proj A)$.
Let
$$ c = \max \{ r_I^M(\Xb) \mid \Xb \neq 0 \}. $$
For $\Yb \in K^b(\proj A)_f$ define
$$\eta(\Yb) = \lim_{n \rt \infty} \frac{c!}{n^c} \psi_{\Yb}^{M, I}(n). $$
Clearly $\xi(\Yb) \in \Z_{\geq 0}$.
Furthermore if $\Yb \cong \Zb$ then clearly $\xi(\Yb) = \xi(\Zb)$. Thus we have a function $\xi \colon \I(K^b_f(\proj A)) \rt \Z$ where $\I(K^b_f(\proj A))$ denotes the set of isomorphism classes of objects in $K^b_f(\proj A)$.

Claim : $\eta$ is a weak triangle function on $K^b_f(\proj A)$.

Assume the claim for the time-being. By \ref{ker-lemma} $\ker \eta$ is a thick subcategory of $K^b_f(\proj A)$. Let $\Xb$ be such that $r_I^M(\Xb) = c$. Then $\eta(\Xb) > 0$. So
$\Xb \notin \ker \eta$. Thus $\ker \eta \neq K^b(\proj A)$. By \cite[Lemma 1.2]{NK} it follows that $\ker \eta = 0$. Thus $r_I^M(\Yb) = c$ for any $\Yb \neq 0$ in $K^b_f(\proj A)$.

It remains to prove the claim. The first four properties of definition in \ref{t-f} are trivial to verify. Let $\Xb \xrightarrow{f} \Yb \rt  \Zb \rt \Xb[1]$ be a triangle in $K^b_f(\proj A)$. Then $Z \cong \cone(f)$ and we have an exact sequence in $C^b(\proj A)$
$$ 0 \rt \Yb \rt \cone(f) \rt \Xb[1] \rt 0.$$
As $\Xb^i$ are free $A$-modules we have an exact sequence for all $n \geq 1$,
$$ 0 \rt \Yb \otimes M/I^nM \rt \cone(f)\otimes M/I^nM  \rt \Xb[1]\otimes M/I^nM  \rt 0. $$
Taking homology we have
\[
\psi_{\Zb}^{M, I}(n) \leq  \psi_{\Yb}^{M, I}(n)  +  \psi_{\Xb[1]}^{M, I}(n)
\]
for all $n \geq 1$. It follows that
\[
\eta(\Zb) \leq \eta(\Yb) + \eta(\Xb[1]) = \eta(\Yb) + \eta(\Xb).
\]
Thus $\eta$ is a weak triangle function on $K^b_f(\proj A)$.
\end{proof}
Next we give
\begin{proof}[Proof of Corollary \ref{reg-cor}]
By Theorem \ref{main} we have that there exists $c$ with $r_I^M(\Xb) = c$ for any non-zero $\Xb \in K^b_f(\proj A)$. Let $L$ be a non-zero finite length $A$-module with finite projective dimension. Let $\Yb$ be a minimal projective  resolution of $L$. Then $\Yb \in K^b_f(\proj A)$ and is non-zero. It follows that $r_I^M(\Yb) = c$. Observe that $r_I^M(\Yb) = t^I_M(L)$.
Set $\Yb^* = \Hom_A(\Yb, A)$. Note that $\Yb^* \in K^b_f(A)$ and is non-zero. Also observe
\[
\Ext^*_A(L, M/I^nM) = H^*(\Hom_A(\Yb, M/I^nM) \cong H^*( \Yb^*\otimes_A M/I^nM).
\]
Therefore
\[
e^I_M(L) = r_I^M(\Yb^*) = c.
\]
The result follows.
\end{proof}
\section{Proof of Theorem \ref{max-dim}}
In this section we assume $(A, \m)$ is local ring, $M$ is an $A$-module with $\dim M > 0$ and $I$ is an $\m$-primary ideal. In this section we give proof of Theorem \ref{max-dim}. We first discuss the invariant $r_I^M(A)$ under base change.

\s\label{bc} \emph{Base change:}\\
(1) We first consider a flat base change $A \rt B$ where $(B, \n)$ is local and $\n = \m B$. We claim that $r_I^M(A) = r_{IB}^{M\otimes_A B}(B)$.

 In this case we first observe that if $E$ is an $A$-module of finite length then $\ell_B(E\otimes_A B) = \ell_A(E)$. Also if $\Xb $ is a bounded complex of $A$-modules with finite length cohomology then $\Xb\otimes_A B$ is a bounded complex of $B$-modules with finite length cohomology and $\ell_B(H^*(\Xb\otimes B) = \ell_A(H^*(\Xb))$. If  $\Yb \in K^b_f(\proj A)$ then $\Yb \otimes_A B \in K^b_f(\proj B)$. Let $\Yb \in K^b_f(\proj A)$ be non-zero.
 Set
 \[
\psi_{\Yb, A}^{M, I}(n) = \sum_{i \in \Z}\ell_A(H^i(\Yb \otimes M/I^nM)), \quad  \text{for $n \geq 1$}.
\]
Then
\begin{align*}
  \psi_{\Yb\otimes_A B, B}^{M\otimes_A B, IB}(n) &= \sum_{i \in \Z}\ell_B(H^i(\Yb \otimes_A B \otimes_B (M/I^nM \otimes_A B))  \\
   &= \sum_{i \in \Z}\ell_B(H^i((\Yb \otimes_A  M/I^nM) \otimes_A B))  \\
   &= \psi_{\Yb, A}^{M, I}(n)
\end{align*}
It follows that degree of the function $\psi_{\Yb, A}^{M, I}(n)$ is equal to degree of $\psi_{\Yb\otimes_A B, B}^{M\otimes_A B, IB}(n)$. The result follows.

(2) If $(Q, \n) \rt (A, \m)$ is a surjective ring homomorphism and if $J$ is any $\n$-primary ideal in $Q$ with $JA = I$ then
$r_I^M(A) = r_J^M(Q)$. To see this, if  $\Yb \in K^b_f(\proj Q)$ then $\Yb \otimes_Q A \in K^b_f(\proj A)$. Let $\Yb \in K^b_f(\proj Q)$ be non-zero.
Set
 \[
\psi_{\Yb, Q}^{M, J}(n) = \sum_{i \in \Z}\ell_Q(H^i(\Yb \otimes_QM/J^nM), \quad  \text{for $n \geq 1$}.
\]
Then
\begin{align*}
  \psi_{\Yb\otimes_Q A, A}^{M, I}(n) &= \sum_{i \in \Z}\ell_A(H^i(\Yb \otimes_Q A \otimes_A M/I^nM )  \\
   &= \sum_{i \in \Z}\ell_Q(H^i((\Yb\otimes_Q  M/J^nM )  \\
   &= \psi_{\Yb, Q}^{M, J}(n)
\end{align*}
The result follows.

(3) If $\q \subseteq \ann_A M$ then note that $M$ can be considered as a $C = A/\q$-module. Set $J = (I+ \q/\q)$. Then $J$ is primary to the maximal ideal of $C$. Then
$r_I^M(A) = r_J^M(C)$. The proof of this is similar to (2).

We now give
\begin{proof}[Proof of Theorem \ref{max-dim}]
We first do the following base-changes:
\begin{enumerate}
  \item If the residue field of $A$ is finite then we set $B = A[X]_{\m A[X]}$ then $(B,\n)$ is a flat extension of $A$ with $\m B = \n$ and the residue field of $B$ is $k(X)$ is infinite. So we replace $M$ by $M \otimes_A B$ and $I$ by $IB$ (see \ref{bc}(1)).
  \item We then complete $A$ (see \ref{bc}(1)).
  \item By (1),(2) we assume $A$ is complete with an infinite residue field. Let $A$ be a quotient of a regular local ring $Q$. Then we can replace $A$ by $Q$, (see \ref{bc}(2)).
  \item By (3) we can assume $A$ is regular local with infinite residue field. We note $a = \grade(\ann M) = \height \ann M$. Choose $y_1, \ldots, y_a \in 
  \ann M $ an $A$-regular sequence. By \ref{bc}(3)
  we can replace $A$ with $A/(y_1,\ldots, y_a)$.
\end{enumerate}
Thus we can assume $A$ is Cohen-Macaulay with infinite residue field and $\dim A = \dim M  > 0$. Let $d = \dim A$ and let $\bx = x_1, \ldots, x_d$ be a maximal $M \oplus A$-superficial sequence with respect to $I$. Then as $\bx$ is an $A$-superficial sequence with respect to $I$ it is an $A$-regular sequence. Let $\Kb$ be the Koszul complex on $\bx$. Then $\Kb \in K^b_f(\proj A)$. We also note that as $x_1$ is $M$-superficial with respect to $I$ there exists $c$  and $n_0$ such that $(I^{n}M \colon x_1)\cap I^cM = I^{n-1}M$ for all $n \geq n_0$.

 Set
 \[
\psi_{\Kb, A}^{M, I}(n) = \sum_{i \in \Z}\ell_A(H^i(\Kb \otimes M/I^nM), \quad  \text{for $n \geq 1$}.
\]
and let $r$ be its degree. By \ref{deg-bound},  $r \leq d - 1$.
We note that
$$H^d(\Kb \otimes M/I^nM) = \frac{I^nM \colon \bx}{I^nM} \supseteq \frac{(I^nM \colon \bx)\cap I^cM}{I^nM}  = \frac{I^{n-1}M}{I^nM} \ (\text{for $n \geq n_0$}).$$
So $\psi_{\Kb, A}^{M, I}(n) \geq \ell(I^{n-1}M/I^nM)$ for all $n \geq n_0$. So $r \geq d -1$. Thus $r = d -1$.
By Theorem \ref{main} it follows that $r_I^M = r = d -1$.
\end{proof}
\section{Proof of Theorem \ref{m-prim-bound}}
In this section we give a proof of Theorem \ref{m-prim-bound}. We need the following well-known result.
Suppose $\dim E > 0$. Then there exists $x \in \m$ such that $(0 \colon_E x)$ has finite length and $\dim E/xE = \dim E - 1$.

We now give give
\begin{proof}[Proof of Theorem \ref{m-prim-bound}]
By \ref{deg-bound}
it suffices to consider the case when $\dim H^*(\Xb \otimes M) \leq \dim M -1$.

We first consider the case when $\dim H^*(\Xb \otimes M) = 0$. We prove the result by inducting on $\dim H^*(\Xb)$. If $\dim H^*(\Xb) = 0$ then the result follows from Theorem \ref{max-dim}. If $\dim H^*(\Xb) > 0$ then choose $x$ such that map $ H^*(\Xb) \xrightarrow{x} H^*(\Xb)$ has finite length kernel and $\dim H^*(\Xb)/xH^*(\Xb) = \dim H^*(\Xb) -1$.
Consider the triangle $\Xb \xrightarrow{x} \Xb \rt \Yb \rt \Xb[1]$.  By taking long exact sequence of homology it follows that $\dim H^*(\Yb)  = \dim H^*(\Xb) - 1$. Furthermore note $\Zb = \cone(x, \Xb)  \cong \Yb$. We have an exact sequence
\[
0 \rt \Xb \rt \Zb \rt \Xb[1] \rt 0.
\]
As $\Xb^i$ is free for all $i$ we have an exact sequence for all $n \geq 0$
\[
0 \rt \Xb\otimes M/I^nM \rt \Zb \otimes M/I^n M \rt \Xb[1]\otimes M/I^nM \rt 0,
\]
and
\[
0 \rt \Xb\otimes M \rt \Zb \otimes M \rt \Xb[1]\otimes M \rt 0.
\]

By considering later short exact sequence of complexes,  we get by looking at long exact sequence in homology that $\dim H^*(\Zb\otimes M)  = 0$.  So by induction hypothesis $s_I^M(\Yb) = \dim M -1$.
By considering all $n \geq 1$ and summing all $i$ we get
\[
\psi_{\Yb}^{M, I}(n) \leq 2 \psi_{\Xb}^{M, I}(n)
\]
It follows that $s_I^M(\Xb) \geq s_I^M(\Yb) = \dim M -1$. But $s_I^M(\Xb) \leq \dim M -1$. The result follows.

We now assume $ 0 < a = \dim H^*(\Xb \otimes M) \leq \dim M -1$ and the result is proved for complexes $\Zb$ with $\dim H^*(\Zb \otimes M) = a -1$. Choose $x$ such that map $ H^*(\Xb \otimes M) \xrightarrow{x} H^*(\Xb\otimes M)$ has finite length kernel and $\dim H^*(\Xb \otimes M)/xH^*(\Xb \otimes M) = \dim H^*(\Xb \otimes M) -1$. Consider the triangle $\Xb \xrightarrow{x} \Xb \rt \Yb \rt \Xb[1]$. Note $\Zb = \cone(x, \Xb)  \cong \Yb$. We have an exact sequence
\[
0 \rt \Xb \rt \Zb \rt \Xb[1] \rt 0.
\]
As $\Xb^i$ is free for all $i$ we have an exact sequence for all $n \geq 0$
\[
0 \rt \Xb\otimes M/I^nM \rt \Zb \otimes M/I^n M \rt \Xb[1]\otimes M/I^nM \rt 0,
\]
and
\[
0 \rt \Xb\otimes M \rt \Zb \otimes M\rt \Xb[1]\otimes M \rt 0.
\]

By considering the latter short exact sequence of complexes,  we get by looking at long exact sequence in homology that $\dim H^*(\Zb\otimes M)  = \dim H^*(\Xb\otimes M) - 1$.  So by induction hypothesis $s_I^M(\Yb) = \dim M -1$.
By considering all $n \geq 1$ and summing all $i$ we get
\[
\psi_{\Yb}^{M, I}(n) \leq 2 \psi_{\Xb}^{M, I}(n)
\]
It follows that $s_I^M(\Xb) \geq s_I^M(\Yb) = \dim M -1$. But $s_I^M(\Xb) \leq \dim M -1$. The result follows.
\end{proof}
\section{Proof of Corollary \ref{cm-max}}
In this section we give a proof of Corollary \ref{cm-max}. We need the following result:
\begin{lemma}\label{proj-supp} Let $A$ be a \CM \ local ring and let $L$ be a non-zero $A$-module of finite projective dimension.
Then
$$ \dim M\otimes L =  \dim \Ext^*_A(L, M). $$
\end{lemma}
\begin{proof}
  It is clear that $\Supp (M\otimes L) = \Supp M \cap \Supp L$.
  Thus it follows that $\Supp \Ext^*_A(L, M)  \subseteq \Supp (M \otimes L)$.
  Conversely let $P \in \Supp M \otimes L$. We localize at $P$. So it suffices to prove $\Ext^*(L, M) \neq 0$. By taking a minimal resolution of $L$ it clear that if $c = \projdim L$ then $\Ext^r_A(L,M) \neq 0$. The result follows.
\end{proof}
We now give
\begin{proof}[Proof of Corollary \ref{cm-max}]
Let $\Xb$ be a minimal projective resolution of $L$. Then\\
$t^I_M(L, n) = \ell(H^*(\Xb \otimes M/I^nM))$.  By \ref{m-prim-bound}  It follows that
$$t^I_M(L) =  \max \{ \dim H^*(\Xb \otimes M), \dim M - 1 \}.$$
The result follows as $\dim H^*(\Xb \otimes M) = \dim M \otimes L$.

Set $\Xb^* = \Hom_A(\Xb, A)$. Observe
\[
\Ext^*_A(L, M/I^nM) = H^*(\Hom_A(\Xb, M/I^nM) \cong H^*( \Xb^*\otimes_A M/I^nM).
\]
So
$$e^I_M(L) =  \max \{ \dim H^*(\Xb^* \otimes M), \dim M - 1 \}.$$
Notice $H^*(\Xb^* \otimes M) = \Ext^*_A(L, M)$. The results follows from Lemma \ref{proj-supp}.
\end{proof}

\end{document}